\documentclass{LMCS}

\def\dOi{10(3:25)2014}
\lmcsheading%
{\dOi}
{1--14}
{}
{}
{Mar.~25, 2013}
{Sep.~23, 2014}
{}

\ACMCCS{[{\bf Theory of computation}]: Logic; [{\bf Software and its
      engineering}]: Software notations and tools---General
  programming languages---Language types---Functional languages}


\usepackage{hyperref}

\theoremstyle{plain}

\usepackage{graphicx}
\input xy
\input diagxy
\xyoption{all}

\newcommand{\changenote}[1]{}
\newcommand{\refl}[0]{{\rm ref}}

\newcommand{\domain}[0]{{\sf dom}}
\newcommand{\codomain}[0]{{\sf cod}}
\newcommand{\compose}[0]{{\sf cmp}}

\newcommand{\vf}[0]{{\varphi}}

\begin{document}

\title[Constructing categories and setoids]{Constructing categories and setoids of setoids in type theory}

\author[E.~Palmgren]{Erik Palmgren\rsuper a}	
\address{{\lsuper a}Stockholm University, Department of Mathematics,
  106 91 Stockholm, Sweden.}	
\email{palmgren@math.su.se}  

\author[O.~Wilander]{Olov Wilander\rsuper b}	
\address{{\lsuper b}Stockholm University, Department of Mathematics. Current affiliation: Sj\"oland \& Thyselius, Box 6238, 102 34 Stockholm. }	
\email{olov.wilander@st.se} 
\thanks{{\lsuper b}This work was  supported by  a grant (Dnr 621-2008-5076) from the Swedish Research Council (VR)}	






\begin{abstract}
In this paper we consider the problem of building rich categories of setoids, in standard 
intensional Martin-L{\"o}f type theory (MLTT), and in particular
how to handle the problem of equality on objects in this context.  Any (proof-irrelevant) 
family $F$ of setoids over a setoid $A$ gives rise to a category  ${\mathcal C}(A, F)$ of setoids
with objects $A$.  \changenote{Rewritten rest of abstract}
We may regard the family $F$ as a setoid of setoids, and a crucial issue in this article is to construct
rich or large enough such families. Depending on closure conditions of $F$, the category ${\mathcal C}(A, F)$ has corresponding categorical constructions. 
We exemplify this with finite limits. A very large family $F$ may be obtained from Aczel's model construction of CZF in type theory. It is proved that the category so obtained
is isomorphic to the internal category of sets in this model. Set theory can thus establish (categorical) properties of ${\mathcal C}(A, F)$ which may be used in type theory.
We also show that Aczel's model construction may be extended to include the elements of any setoid as atoms or urelements. As a byproduct we obtain a natural extension of CZF, adding atoms. This extension, CZFU, is validated by the extended model. The main theorems of the paper have been checked in  the proof assistant Coq which is based on MLTT.
A possible application of this development is to integrate set-theoretic and type-theoretic reasoning in proof assistants.
\end{abstract}

\maketitle


\section{Introduction}

Martin-L{\"o}f type theory (MLTT) and its manifestations, in proof assistants such as Agda and Coq, is intended to be a framework
for formalizing (constructive) mathematics on a full scale. It is known that the intensional version of MLTT is sometimes difficult
to employ when formalizing mathematics that depends on having (propositional) equality between sets or setoids. 
This may be troublesome in parts 
of category theory \cite{P,W} where an equality on objects is a standard assumption. A typical situation is when we wish to deal with some  category of sets or setoids on equal footing to other categories. The built-in propositional equalities of type theory, given by the intensional identity types, are not extensional enough for this work without further complications. The root of the problem is that the intensional identity type of MLTT induces a non-trivial groupoid structure on types \cite{HS}. This can be avoided by introducing extra elimination axioms
like the K-axiom of Streicher \cite{St93}, the Uniqueness of Identity Proofs axiom, or weaker axioms \cite{W}. Adding these axioms is, however, an unsatisfactory solution 
according to the general philosophy of MLTT, where the elimination rule is supposed to be generated by the introduction rule.

In this paper we consider solutions to this problem within the standard intensional version of MLTT with one universe and W-types. In \changenote{Added three sentences}
Section 4.1  we recall that any (proof-irrelevant) 
family $F$ of setoids over a setoid $A$ gives rise to a category  ${\mathcal C}(A, F)$ of setoids
with objects $A$. We may regard the family $F$ as a setoid of setoids, and a crucial issue here is to construct
rich or large enough such families. Depending on closure conditions of $F$, the category ${\mathcal C}(A, F)$ has corresponding categorical constructions.   
A \changenote{Added paragraph}
first solution is to explicitly construct $F$ such that it is closed under particular constructions, here exemplified by finite limits (Section 4.2). A smoother solution is however to employ a universe $V$ of iterative sets that forms a model of Aczel-Myhill constructive  set theory (CZF), and consider the category of setoids that the 
sets of $V$ induces. This turns out to be a well-behaved category of setoids, which is isomorphic to the internal category of sets of $V$ (Theorem \ref{main}). Theorem \ref{main} allows passage between the setoids of type theory and the sets of $V$.  \changenote{Added sentence} The model and the theorem have been formalized in the proof assistant Coq, and give, in principle, a method for  formalizing further category-theoretic results in Coq that depend on a good category of sets.  This makes it possible to integrate set-theoretic and type-theoretic reasoning, where set theory has a simpler language which is better adapted to solve certain complications arising from transport functions of families of setoids, and type theory has a more direct computational interpretation. \changenote{Motivation added}

Models of CZF have previously been implemented in systems similar to Coq: in LEGO by Mendler \cite{Me} and in Agda/Alfa by Takeyama \cite{T}.
\changenote{Corrected reference} See also Hickey \cite{Hi} and Yu \cite{Y} for work  done in MetaPRL. However, we add a new twist here by allowing urelements or atoms in the model, and importantly, giving the relation to setoids, via the notion of  a $V$-representable setoid (Section
\ref{V_represent}). Our formalized model moreover allows us to embed an arbitrary setoid $M$ in a 
CZF-universe $V(M)$.  As a bonus of the construction $V(M)$, we obtain a model of CZF with atoms (elements of $M$), which is formalized as a first-order theory CZFU (Section \ref{CZFU_section}).  We end by some remarks on the formalization in Coq (Section \ref{coq_implementation}).

\section{Setoids} \label{setoid_sec}

In the following we freely use the propositions-as-types principle in the logical arguments. Thus we may speak of a proof $q$ of a proposition $Q$, meaning that $q$ is an element of type $Q$ which is written $q:Q$ as usual. In our implementation in Coq this  corresponds to
avoiding the built-in type Prop and using Set or Type for propositions. (See Section \ref{coq_implementation}.)

Recall that a {\em setoid} $A=(|A|,=_A)$ is a type $|A|$ with an equivalence
relation $=_A$. We denote the constructions associated
with proofs of reflexivity, symmetry and transitivity as follows
\begin{itemize}
\item[] $\refl(x) : x=_A x$ \quad ($x:|A|$)
\item[] $p^{-1}: y=_A x$ \quad ($x,y:|A|, p: x =_A y$)
\item[] $q \circ p:x=_A z$ \quad ($x,y,z:|A|, p: x=_A y, q: y =_A z$)
\end{itemize}
We shall often write $x \in A$ for $x:|A|$ to simplify notation.
For setoids $A$ and $B$, an {\em extensional function}
$f:A \to B$ is a pair $f=(|f|,{\rm ext}_f)$ where $|f|:|A| \to |B|$
and ${\rm ext}_f$ is a
proof-object for extensionality of the operation $|f|$, that is that
$$(\forall x,y \in A)(x =_A y \Longrightarrow |f|(x) =_B |f|(y)).$$
We write $f(x)$ for $|f|(x)$.

For setoids $A$ and $B$ denote by ${\rm Ext}(A,B)$ the setoid of extensional functions from $A$ to $B$, with point-wise
equality ($=_{\rm ext}$) as equivalence relation. \changenote{Corrected typo}
The setoids and extensional functions form an E-category, which shall be named
{\bf Setoids} here.  We recall that an E-category  $\mathcal C$ has a type of objects with no equality assumed between them. The morphisms, denoted
${\rm Hom}_{\mathcal C}(A,B)$, from object $A$ to $B$ is a setoid and the composition operation
$$\circ : {\rm Hom}_{\mathcal C}(B,C) \times {\rm Hom}_{\mathcal C}(A,B) \to {\rm Hom}_{\mathcal C}(A,C)$$
is an extensional function. The usual laws for composition and identity are supposed to be satisfied.

\begin{exa} \label{Fecat}
Let $F(x) \; (x:S)$ be a family of setoids indexed by a type $S$. Then an E-category ${\mathcal E}(S,F)={\mathcal E}$ of setoids can be formed 
whose type of objects is $S$ and where
$${\rm Hom}_{\mathcal E}(a,b) = {\rm  Ext}(F(a),F(b)).$$
\end{exa}

\begin{rem} \changenote{New remark} 
The E-category  {\bf Setoids} considered here (and elsewhere) is closely related to the exact completion of a syntactic category arising from
type theory (cf. \cite{Ca}). The match is not precise since the E-category is internal to type theory, whereas the exact completion takes place on the meta-level.
Barthe {\em et al.} \cite{BCP} compare several variants of the notion of setoid and their formalization in Coq. In their terminology we use {\em total setoids} but employ
{\bf Set}-valued equivalence relations instead of  their {\bf Prop}-valued ones.
\end{rem}

\section{Families of setoids} \label{setoidfam_sec}

A good notion of a family of setoids over a setoids is the following (compare the discussion in \cite{P}).
A {\em proof irrelevant family $F$ of setoids over a setoid $A$} --- or just {\em family of setoids} --- consists of a setoid
$F(x) = (|F(x)|, =_{F(x)})$  for each $x\in A$, 
and for $p: (x=_A y)$ an extensional function $F(p) \in {\rm Ext}(F(x),F(y))$
which satisfies the conditions (F1) -- (F3) below. \changenote{Added that $A$ is a setoid.}
\begin{enumerate}[label=(F\arabic*)]
\item $F(\refl(x))=_{\rm ext} {\rm id}_{F(x)}$ for  $x \in A$.

\item $F(p) =_{\rm ext} F(q)$ for $p,q : x=_A y$ and $x,y \in A$. This is the proof-irrelevance condition, since
$F(p)$ does not depend on $p$.

\item $F(q) \circ F(p) =_{\rm ext} F(q \circ p)$ for 
$p: x =_A y$, $q: y =_A z$ and $x,y,z \in A$.
\end{enumerate}
The function $F(p)$ is sometimes called a {\em transport function.}
Because of condition (F2), condition (F1) can be replaced by (F1') 
$$(\forall x \in A)(\forall p: x=_A x)F(p) =_{\rm ext} {\rm id}_{F(x)}$$ and condition (F3) can be replaced by (F3')
$$(\forall x,y,z \in A)(\forall p: x=_A y)(\forall q: y=_A z)(\forall r: x=_A z)F(q) \circ F(p) =_{\rm ext} F(r).$$
We shall sometimes use the notation $x \cdot p$ for $F(p)(x)$ when $F$ is clear from the context.

\medskip
As can  be seen from (F1) -- (F3) a family $F$ may be regarded as a functor (or rather E-functor) from the
discrete E-category $A^\#$, induced by $A$, to {\bf Setoids}.

\section{From  families to categories of setoids} \label{cat_from_setoidfam_sec}

It is well-known that the E-category of setoids in Martin-L\"of type theory forms a locally 
cartesian closed (LCC)  category; see \cite{H94}. It can moreover be 
shown to be a pretopos  with further properties \cite{MP}.  In fact, one can straightforwardly verify in Coq 
(see for instance \cite{PW12}) that the E-category of setoids
forms an LCC pretopos. For categories of setoids with equality on objects  the constructions are more delicate and 
this is the subject of this and the next section.

Categories can be presented in an essentially algebraic way; cf.\ \cite{MacL}. \changenote{Added} 
This is a useful formulation especial for doing category theory internally to a category, but also in set theory.
A  {\em (small) category ${\mathcal C}$} is a triple of classes (sets) ${\mathcal C}_0$, ${\mathcal C}_1$, ${\mathcal C}_2$ consisting of  {\em objects,} {\em arrows} and {\em composable arrows,}
equipped with class functions (functions) ${\sf id}: {\mathcal C}_0 \to {\mathcal C}_1$, ${\sf dom}, {\sf cod}: {\mathcal C}_1 \to {\mathcal C}_0$ and
${\sf cmp}, {\sf fst}, {\sf snd}: {\mathcal C}_2 \to {\mathcal C}_1$ that satisfy the axioms
\begin{enumerate}[label=\,\arabic*.\enspace]
\item[1.] ${\sf dom}({\sf id} (x)) = x$,
\item[2.] ${\sf cod}({\sf id} (x)) = x$,
\item[3.] ${\sf dom}({\sf cmp}(u)) =
                    {\sf dom} ({\sf fst}(u))$, 
\item[4.] ${\sf cod}({\sf cmp}(u)) =
                    {\sf cod} ({\sf snd}(u))$,
\item[$4\frac{1}{2}.$] ${\sf cod}({\sf fst}(u)) =
                    {\sf dom} ({\sf snd}(u))$,
\end{enumerate}
and
\begin{enumerate}[label=\,\arabic*.\enspace]

\item[5.] ${\sf fst}(u) = {\sf fst}(v),
                 {\sf snd}(u) ={\sf snd}(v) \implies u = v$,
\item[6.] ${\sf dom}(f) = {\sf cod}(g) \implies
    \exists u \in {\mathcal C}_2 ({\sf snd}(u) = f  \land {\sf fst}(u) = g)$,
\item[7.] ${\sf fst}(u) = {\sf id}(y) \implies
              {\sf cmp}(u) = {\sf snd}(u)$,
\item[8.] ${\sf snd}(u) = {\sf id}(x) \implies
              {\sf cmp}(u) = {\sf fst}(u)$,
\item[9.] ${\sf fst}(w) = {\sf fst}(v), {\sf snd}(v) = {\sf fst}(u), 
       {\sf snd}(u) = {\sf snd}(z),{\sf snd}(w) = {\sf cmp}(u), 
       {\sf cmp}(v) = {\sf fst}(z) \implies {\sf cmp}(w) = {\sf cmp}(z)$.
\end{enumerate}
The category {\bf Set} of sets and functions can be constructed in the standard way in ZF or CZF 
using these operations, and its categorical properties developed, following e.g. \cite{BvdB}.

We can simply obtain a type-theoretic definition by replacing sets and classes with setoids (with respective size restrictions) and functions by extensional functions.
 
A {\em functor} $F: {\mathcal B} \to {\mathcal C}$ is a triple of extensional functions $F_k: {\mathcal B}_k \to {\mathcal C}_k$, $k=0,1,2$, such
that all operations of the categories are preserved, that is
\begin{itemize}
\item[] $F_1 \circ {\sf id} = {\sf id} \circ F_0$,
\item[] $F_0 \circ {\sf dom} = {\sf dom} \circ F_1$,
\item[] $F_0 \circ {\sf cod} = {\sf cod} \circ F_1$,
\item[] $F_1 \circ {\sf fst} = {\sf fst} \circ F_2$,
\item[] $F_1 \circ {\sf snd} = {\sf snd} \circ F_2$,
\item[] $F_1 \circ {\sf cmp} = {\sf cmp} \circ F_2$. 
\end{itemize}
The axioms 1 -- 9 take a more familiar form if we rewrite them using the composition predicate ${\rm Comp}(f,g,h)$ (or $f \circ g  \equiv h$) defined by
$$(\exists u \in {\mathcal C}_2)({\sf fst}(u) = g \land {\sf snd}(u) = f \land
                {\sf cmp}(u) = h).
$$

\begin{rem}
Any category ${\mathcal C}$ may be viewed as an E-category $\overline{{\mathcal C}}$ by ignoring the equality on objects
and defining ${\rm Hom}_{\overline{{\mathcal C}}}(a,b)$ to be the setoid
$$((\Sigma f \in {\mathcal C}_1)[{\sf dom}(f) = a \land {\sf cod}(f) = b], \sim)$$
where $(f,p) \sim (f',p)$ iff $f =_{{\mathcal C}_1} f'$. Composition and identity are then defined in the obvious way
using the axioms above.

We may even consider ${\rm Hom}$ as a proof-irrelevant family over the setoid ${\mathcal C}_0 \times {\mathcal C}_0$. This gives an equivalent notion of category which is perhaps more natural in the type-theoretic language. \changenote{Added}
\end{rem}

\subsection{Construction of a category of setoids} \label{catfromfam}

We recall the following construction from \cite{P} and \cite{W}. 
Any family $F$ of setoids over a setoid $A$ gives rise to a category of setoids   
${\mathcal C}={\mathcal C}(A,F)$ in the following way. The objects are given by
the index setoid ${\mathcal C}_0=A$, and are thus equipped with equality, and the
setoid of arrows ${\mathcal C}_1$  is
$$((\Sigma x,y: |A|){\rm Ext}(F(x),F(y)),\sim)$$ 
where two arrows are equal $(x,y,f) \sim (u,v,g)$ if, and only if, there are proof objects $p: x=_A u$ and $q: y=_A v$  such that
the diagram
$$\bfig\square[F(x)`F(y)`F(u)`F(v);f`F(p)`F(q)`g]\efig$$
commutes, or equivalently
$$(\forall t \in F(x))[f(t) \cdot q =_{F(v)}  g(t \cdot p)].$$
(Note that $F(p)$ and $F(q)$ are independent of $p$ and $q$.) 
The domain and codomain maps $\domain:  {\mathcal C}_1  \rightarrow  {\mathcal C}_0 $ and  $\codomain:  {\mathcal C}_1  \rightarrow {\mathcal C}_0 $ are given by
$\domain(x,y,f) = x$ and $\codomain(x,y,f) = y$. The setoid ${\mathcal C}_2$ of composable maps is then
$$((\Sigma h, k : |{\mathcal C}_1|)[ \codomain(h) =_{{\mathcal C}_0} \domain(k)], \approx)$$
where $(h,k,p) \approx (h',k',p')$ if and only if $h \sim h'$ and $k \sim k'$. The composition map $\compose: {\mathcal C}_2 \to {\mathcal C}_1$ is given
by 
$$\compose((x,y,f),(u,v,g),p) =_{\rm def} (x,v, g \circ F(p) \circ f).$$
Furthermore, let
$${\sf fst}((x,y,f),(u,v,g),p) =_{\rm def} (x,y,f) \qquad  {\sf snd}((x,y,f),(u,v,g),p) =_{\rm def} (u,v,g).$$
It is straightforward to verify
\begin{thm} \label{cafthm}
If $F$ is a family of setoids over a setoid $A$, then  ${\mathcal C}={\mathcal C}(A,F)$  is a small category. \qed
\end{thm}

\begin{lem} \label{compchar}
In the category ${\mathcal C}(A,F)$ the composition predicate ${\rm Comp}$ may be characterized as follows
$${\rm Comp}((c,d,g),(a,b,f),{\bf h}) \Longleftrightarrow (\exists r: b =_A c)  (a,d,g \circ F(r) \circ f) \sim {\bf h}.$$
If $b$ and $c$ are definitionally equal, then $F(r)$ is the identity map.  \qed
\end{lem}

\subsection{Closure conditions on families and categories} \changenote{Made into new subsection}

An important property of the category of sets is that the terminal object (i.e.\ the singleton set) generates
the category. In such categories it is possible to interpret the internal logic in terms of elements; see \cite{P12}.
We recall some definitions before proving that ${\mathcal C}(A,F)$ has the same property. \changenote{Added introductory sentences}

Let ${\mathcal D}$ be a category with terminal object $1$. Recall that an {\em element} of an object $X$ is
an arrow $x:1 \to X$. An arrow $f:X \to Y$ of the category is here called
{\em onto } if for every $y:1 \to Y$, there is some $x:1 \to X$ with $f \circ x = y$. The arrow is as usual {\em mono}
if for any $g,h:U \to X$ in ${\mathcal D}$, $fg=fh$ implies $g=h$. \changenote{Added sentence to explain mono.} If each arrow $f:X \to Y$ in $\mathcal D$
that is both onto and mono, is also an isomorphism, then we say that $1$ is a {\em strong generator} for $\mathcal D$. 
For a family $F$ of setoids over $A$, we say that {\em $c \in A$ represents a setoid $C$} if $F(c)$ is isomorphic to $C$.
We also say that {\em $F$ contains $C$ }(up to isomorphism).\changenote{Added two sentences}
The category ${\mathcal C}(A,F)$ has a strong generator whenever the family $F$ contains the terminal object.
This follows from the straightforwardly proven result. Note that part (d) uses the  type-theoretic choice principle coming from $\Sigma$-elimination.
\changenote{Changed Lemma to Proposition}

\begin{prop} Let $F$ be a family of setoids indexed by the setoid $A$, and suppose that $c \in A$ represents the terminal setoid.
Then 
\begin{enumerate}[label=\({\alph*}] \label{generator}
\item $c$ is the terminal object in ${\mathcal C}(A,F)$.
\item If $(a,b,f)$ is an arrow of ${\mathcal C}(A,F)$ then it is mono 
if and only if $f:F(a) \to F(b)$ is injective.
\item If $(a,b,f)$ is an arrow of ${\mathcal C}(A,F)$ then it is onto 
if and only if $f:F(a) \to F(b)$ is surjective.
\item The terminal object of  ${\mathcal C}(A,F)$ is a strong generator for the category. \qed
\end{enumerate}
\end{prop}

If the family $F$ is a universe, we get a category ${\mathcal C}(A,F)$ with closure conditions depending on the type-theoretic closure conditions  of the universe. In \cite{MP} it was shown that by letting $A,F$ be a particular universe of $U$-small setoids, the category is a locally cartesian closed pretopos with $W$. By a such a universe we mean that,  for each $a \in A$, $F(a)$ is a setoid where both the underlying type $|F(a)|$ and the truth-values of $x=_{F(a)} y$ are in the type theoretic universe $U$. \changenote{Added explanation}
However, the construction of $A$ and $F$ in that paper used constructions going outside standard intensional type theory, in fact, a tacit assumption was made of a principle (see \cite[Theorem 5.2]{P}) which is equivalent to Uniqueness of Identity Proofs, which, in turn, is false in the groupoid model. In \cite{W} a somewhat weaker axiom is proposed, which may possibly let the constructions of \cite{MP} go through.
We have constructed (in Coq) a graded universe of setoids $A_{\omega},F_{\omega}$, with no transfinite types, but closed under grade bounded $\Pi$ and $\Sigma$, as well as sums and coequalizers, to be able to mimic constructions of categorical universes in extensional type theory \cite{Ma}. However the expected categorical properties  of ${\mathcal C}(A_{\omega},F_{\omega})$ have turned out quite difficult to verify formally. 
\changenote{Deleted a sentence} In the next subsection we present instead a method to taylor particular categorical universes which is more manageable.

\subsection{Direct construction of categorical universes}
 \changenote{added new section about alternate construction}

We present a method for constructing categories of setoids closed under particular constructions, and exemplify with
the construction of pullbacks.

Let $S$ be the inductive type defined by the rules
\begin{equation} \label{stages}
\frac{}{{\sf b}:S} \qquad \frac{i:S \quad j:S \quad k:S}{{\sf p}(i,j,k):S}.
\end{equation}
The identity on this type is decidable and satisfies
$$I(S,{\sf p}(i,j,k),{\sf p}(i',j',k')) \Longleftrightarrow I(S,i,i') \land I(S,j,j') \land I(S,k,k').$$
This forms the setoid of {\em construction stages}. The symbol ${\sf b}$ signifies the basic stage.

Let $G$ be any family of setoids indexed by a setoid $B$.
Then define by recursion on $s\in S$, $A_s$ and $F_s$, such that $A_s$ is a setoid and $F_s$ is family of setoids on $A_s$. 

Let $A_{\sf b} = B$ and $F_{\sf b} =G$.

Let 
\begin{equation}
\begin{split}
A_{{\sf p}(i,j,k)} = & \Bigl((\Sigma a:A_i)(\Sigma b:A_j)(\Sigma c:A_k)(\Sigma d:A_k) \\
 &\qquad \bigl[(c=_{A_i}d) \times (F_i(a) \to F_k(c))\times (F_j(b) \to F_k(d))\bigr], \sim\Bigr)
\end{split}
\end{equation}
where
\begin{equation}
(a,b,c,d,q,f,g) \sim (a',b',c',d',q',f',g')
\end{equation}
is given by
\begin{equation} \label{speq1}
\begin{split}
 &(\exists p_1:a=_{A_i} a')(\exists p_2:b=_{A_j} b')(\exists p_3:c=_{A_k} c')(\exists p_4:d=_{A_k} d')\\
& \qquad F_k(p_3) \circ f = f' \circ F_i(p_1) \land F_k(p_4) \circ g= g' \circ F_j(p_2)
\end{split}
\end{equation}
and then
$$F_{{\sf p}(i,j,k)}(a,b,c,d,p,f,g) = \Bigl((\Sigma x:F_i(a))(\Sigma y:F_j(b))\bigl[F_k(p)(f(x)) =_{F_k(d)} g(y))\bigr], \approx \Bigr). $$
where
$$(x,y,r) \approx (x',y',r') \Longleftrightarrow_{\rm def} x=_{F_i(a)} x' \land y =_{F_j(b)} y'.$$
For $(p_1,p_2,p_3,p_4,q_1,q_2):(a,b,c,d,p,f,g) \sim (a',b',c',d',p',f',g')$, define
$$F_{{\sf p}(i,j,k)}(p_1,p_2,p_3,p_4,q_1,q_2): F_{{\sf p}(i,j,k)}(a,b,c,d,p,f,g) \to F_{{\sf p}(i,j,k)}(a',b',c',d',p',f',g')$$
by letting
$$F_{{\sf p}(i,j,k)}(p_1,p_2,p_3,p_4,q_1,q_2)(x,y,r) =(F_i(p_1)(x),F_j(p_2)(y),r')$$
where $r'$ is some proof of 
$$F_k(p') (f'(F_i(p_1)(x)))=_{F_k(d')}g'(F_j(p_2)(y))$$
(that can be obtained from (\ref{speq1})). It is straightforward to check that $F_{{\sf p}(i,j,k)}$ is a family of setoids over $A_{{\sf p}(i,j,k)}$.
Moreover the following is a pullback square in the E-category of setoids
\begin{equation} \label{sppb}
\bfig
\square<1000,500>[F_{{\sf p}(i,j,k)}(a,b,c,d,p,f,g)`F_j(b)`F_i(a)`F_k(d);\pi_2`\pi_1`g`F_k(p) \circ f ],
\efig
\end{equation}
where $\pi_1(x,y,r)=x$ and $\pi_2(x,y,r)=y$.

Define using $I$-elimination, for $p:I(A,s,s')$,
$${\rm transport}_{\lambda s. A_s}(p): A_s \to A_{s'}$$ 
by letting $C(s,s',p)=_{\rm def}A_s \to A_{s'}$ and ${\rm transport}_{\lambda s. A_s}({\sf ref}(s)) = \lambda x:A_s.x$.

Finally we define 
$$A_{\omega} =((\Sigma s:S)A_s, =_{\omega})$$
where 
$$(s,a) =_{\omega} (s',a') \Longleftrightarrow_{\rm def} (\exists p: I(S,s,s')) {\rm transport}_{\lambda s. A_s}(p)(a)=_{A_{s'}} a'$$
and
$$F_{\omega}(s,a)= F_s(a),$$
and further for $(p,q): (s,a) =_{\omega} (s',a')$, we define
$$F_{\omega}(p,q)(x) = F_{s'}(q)(H_{s,s',p}(x)),$$
where
$H_{s,s',p}: F_s(a) \to F_{s'}({\rm transport}_{\lambda s. A_s}(p)(a))$
is obtained by $I$-elimination with $$C(s,s',p)=_{\rm def} F_s(a) \to F_{s'}({\rm transport}_{\lambda s. A_s}(p)(a))$$ and
$H_{s,s,{\sf ref}}(s)= \lambda x:F_s(a). x$.
The identity type of $S$ is decidable, so it enjoys the Uniqueness of Identity Proofs property by Hedberg's theorem \cite{P}. Then one may easily verify that $F_{\omega}$ is a family of setoids over $A_{\omega}$. Furthermore the category ${\mathcal C}= {\mathcal C}(A_{\omega},F_{\omega})$ given by this family has chosen pullbacks, which means that
there are two extensional functions $ {\rm p}_1,{\rm p}_2:{\rm M}({\mathcal C}) \to {\rm Ob}({\mathcal C})$ defined on the setoid of arrows with common codomain
$${\rm M}({\mathcal C})= \{(f,g) \in {\rm Arr}({\mathcal C})^2 : {\rm cod}(f) =_{{\rm Ob}({\mathcal C})} {\rm cod}(g) \}$$
such that for all $(f,g) \in {\rm M}({\mathcal C})$,
$$\bfig\square[\cdot`\cdot`\cdot`\cdot;{\rm p}_2(f,g)`{\rm p}_1(f,g)`g`f]\efig$$
is a pullback. Using these constructions it is now possible to verify:

\begin{thm} \label{haspullbacks}
The category ${\mathcal C}(A_{\omega},F_{\omega})$ has chosen pullbacks. \qed
\end{thm}

We expect that it should be possible to extend the construction above to other properties (e.g.\ LCC pretoposes) by adding new construction stages to (\ref{stages}). The formal verification will probably be quite cumbersome. However it is possible to obtain rich categorical universes that are smoother to construct and verify. In the next section we show that chosing $A$ and $F$ to be induced by the Aczel universe $V$ of iterative sets, the category ${\mathcal C}(A,F)$ gets good categorical properties; see Theorem \ref{main}.

\section{Aczel's iterative sets and setoids}

It is known that the category of sets inside  Constructive Zermelo-Fraenkel set theory (CZF) has good category-theoretic properties \cite{BvdB}. These can preferably be established on basis of the essentially algebraic formulation of categories given in Section 4. \changenote{Added sentence.} Aczel \cite{A}  presented a model of CZF in MLTT. 
This suggests that we may use such models of CZF to build useful categories for type theory. \changenote{Corrected typo}
 The model builds on the iterative conception of set, which is to say,
a set is a, possibly infinite, well-founded tree, and where equality of sets is defined in terms of bisimulation. 

\subsection{Iterative sets with urelements} \label{VM}

We consider here a modification of Aczel's standard model of CZF, to be able to add urelements or atoms.
For a universe $U,T(\cdot)$, and a setoid $M =(|M|,=_M)$ (of urelements), the set-theoretic 
universe $V(M)=V$ is inductively defined by the rules
$$\frac{a:U \quad f:T(a) \to V}{\sup(a,f):V} \qquad \frac{b: |M|}{{\rm atom}(b):V}.$$
The equality $=_V$ is the smallest relation satisfying the two rules

$$\frac{\forall x:T(a).\exists y:T(b). f(x) =_V g(y) \quad  \forall y:T(b).\exists x:T(a). f(x) =_V g(y)}{ \sup(a,f) =_V \sup(b,g) }$$

$$\frac{a =_M b}{{\rm atom}(a) =_V {\rm atom}(b)}.$$
The membership relation is defined by
$$u \in_V \sup(a,f) \Longleftrightarrow \exists x:T(a). u=_V f(x)$$
and declaring $u \in_V {\rm atom}(b)$ to be false. We have $a=_M b$ iff ${\rm atom}(a) =_V {\rm atom}(b)$, so that equality of atoms is exactly that of the setoid.  The standard model is the special case when $M$ is the empty setoid (no atoms).

\medskip
We say that a  setoid {\em $M=(|M|,=_M)$ belongs to the universe $U$} if there is some $m : U$ with $|M|=T(m)$, and some
$e: |M| \to |M| \to U$ such that for all $x,y: |M|$,
$$x=_M y \Longleftrightarrow T(e(x,y)).$$
For such setoids we have:

\begin{lem} \label{smallness}
If $M$ is a setoid which belongs to $U$, then the relations $x =_V y$  and $x \in_V y$ are propositions in $U$.
\end{lem} 

It is crucial that the basic relations $\in$ and $=$ are interpreted as propositions in the universe $U$ in order
to be able to verify that all bounded formulas ($\Delta_0$-formulas) may be used in the separation scheme of CZF.   We will thus  consider $V(M)$ where the setoid
$M$ belongs to $U$.  There is no principal difficulty in extending the construction to finitely many setoids of atoms. \changenote{Added remark about several setoids of atoms.}

\subsection{$V$-representable setoids} \label{V_represent}

We consider here for simplicity only pure sets, thus let $V=V(\emptyset)$.
For each $u:V$ define the setoid $$B(u) = (|B(u)|, =_{B(u)})$$ of elements of $V$ belonging to $u$ by letting
$$|B(u)| = \Sigma z:V. z \in_V u$$
and
\begin{equation} \label{Meq}
(z,p) =_{B(u)} (z',p') \Longleftrightarrow z=_V z'.
\end{equation}
Note that for a set $u=\sup(a,f)$, it holds that $$B(\sup(a,f))\cong (T(a), \sim_f)$$ where
$$x \sim_f x' \Longleftrightarrow f(x) =_V f(x').$$
We define therefore 
$$R(\sup(a,f)) = (T(a), \sim_f).$$
It is thereby easy to find the setoid and its underlying type from the set.
A setoid $A$ is {\em $V$-representable} iff there is some $u:V$ and a bijection $\phi:A \cong R(u)$.
Let $u=\sup(a,f)$ and $v=\sup(b,g)$.
If we examine $${\rm Ext}(R(u), R(v)),$$ the standard construction of the setoid of functions  from 
$R(u)$ to $R(v)$,
it has the underlying type
\begin{equation} \label{fntype}
\Sigma h: T(a) \to T(b).(\forall x,y:T(a) (fx =_V fy \Rightarrow g(hx) =_V g(hy)))
\end{equation}
and equality $\sim$ defined by
$$(h,p) \sim (h',p') \text{ iff } \forall x:T(a). g(hx) =_V g(h'x).$$
Let $F_{u,v}$ denote the type in (\ref{fntype}). Define 
$$\gamma(h,p)= \sup(a, \lambda x. \langle fx, g(hx) \rangle)$$
which gives the graph of the function $h$, when $(h,p) : F_{u,v}$. Suppose that the type $F_{u,v}$ has a code $\varphi_{u,v}$ in $U$ so that
$F_{u,v}=T(\varphi_{u,v})$. Now we can form
$$v^u=\sup(\varphi_{u,v},\gamma),$$
which is the set all of functions from $u$ to $v$. Indeed we have
$$z \in_V v^u \text{ iff }  z \text{ is a total and functional relation from $u$ to $v$},$$
where the latter can be formally expressed as the conjunction of the following statements
$$(\forall t \in V)(t \in_V z \Rightarrow (\exists x,y \in V )(x \in_V u \land y \in_V v \land t=_V \langle x,y\rangle)),$$
$$(\forall x \in V)(x \in_V u \Rightarrow (\exists y \in V)(y \in_V v \land \langle x,y \rangle \in_V z)),$$
$$(\forall x,y,y' \in V)(\langle x,y \rangle \in_V z \land \langle x,y' \rangle \in_V z \Rightarrow y=_V y').$$
Note that these are the interpretations of the corresponding first-order CZF formulas in the structure 
$(V,=_V,\in_V)$. We have the following bijective correspondence
\begin{prop} For any $u=\sup(a,f),v=\sup(b,g) \in V$, there is a bijection
$$\psi: R(v^u) \to {\rm Ext}(R(u),R(v))$$
given by $\psi(h,p)= (h,p)$. \qed
\end{prop}

Actually we have arrived at the standard definition of the function set by analyzing representable sets and
functions.

\subsection{Two isomorphic categories}
\changenote{added emphasis on how the first-order CZF-language is used.}

The internal category of sets in $V$ may be described as follows.
Define the category ${\mathcal V}$ to have as objects ${\mathcal V}_0$  the setoid $V=(V,=_V)$.  The arrows ${\mathcal V}_1$
has as underlying type 
$$\Sigma u \in V. {\rm Isarrow}(u)$$
where ${\rm Isarrow}(u)$ is the predicate definable using CZF formulas
$$\exists a, b, f  \in V. u =_V \langle \langle a,b \rangle, f\rangle \land 
\text{ $f$ is a total and functional relation from $a$ to $b$}.$$
Equality $(u,p) =_{{\mathcal V}_1} (u',p')$ is defined to be $u=_V u'$.
The setoid ${\mathcal V}_2$ of composable arrows has for underlying type
$$\Sigma w \in V. \Sigma u,v \in {\mathcal V}_1. w =_V \langle \pi_1(u),\pi_1(v) \rangle \land {\sf cod}\, u =_v {\sf dom}\, v$$
and its equality is given by $(w,p) \sim (w',p')$ iff $w=_Vw'$.
Again these can be given by straightforward intepretations of first-order CZF-formulas.
Composition ${\sf cmp}$ of arrows is obtained by composition of relations in the usual set-theoretic way.

\begin{thm} \label{V_is_cat} 
${\mathcal V}$ is a category. \qed
\end{thm}

A different category is constructed using the method of Section \ref{catfromfam}.
We extend $R(\cdot)$ to a family of setoids $\bar{R}$ over the setoid $V=(V,=_V)$.
Let $\bar{R}(\alpha) = R(\alpha)$ for $\alpha \in V$. For a proof object $p$
for $\sup(a,f) =_V \sup(b,g)$, or equivalently, for
$$\forall x:T(a).\exists y:T(b). f(x) =_V g(y) \land  \forall y:T(b).\exists x:T(a). f(x) =_V g(y),$$
we thus have
$$ \forall x:T(a).f(x)=_V g(\pi_1(\pi_1(p)(x))) \text{ and }  \forall y:T(b). f(\pi_1(\pi_2(p)(y))) =_V g(y).$$
Let $\bar{R}(p)(x) =  \pi_1(\pi_1(p)(x))$. This defines an extensional function 
$$\bar{R}(p): R(\sup(a,f)) \to R(\sup(b,g)).$$

\begin{lem} $\bar{R}$ is a family of setoids over $(V,=_V)$.
\end{lem}
\begin{proof}  The function $\bar{R}(p): \bar{R}(\sup(a,f)) \to \bar{R}(\sup(b,g))$, 
is independent of $p$. Indeed, if $p, p'$ are arbitrary and $x \sim_f x'$, then
$$g(\bar{R}(p)(x)) =_V f(x) =_V f(x') =_V g(\bar{R}(p')(x).$$
This verifies (F2).
If $p: \sup(a,f) =_V \sup(a,f)$, then $f(\bar{R}(p)(x))=_V f(x)$, so  $\bar{R}(p)(x) \sim_f x$. Hence $\bar{R}(p)$ is the identity, 
and (F1) is clear. Finally, we check (F3').
Suppose we have three proof objects $p:\sup(a,f) =_V \sup(b,g)$, $q:\sup(b,g) =_V \sup(c,h)$ and 
$r:\sup(a,f) =_V \sup(c,h)$.
Expanding as above we have $g(\bar{R}(p)(x))=_V f(x)$ and $h(\bar{R}(q)(y))=_V g(y)$ for all $x$ and $y$. Thus
$$h(\bar{R}(q)(\bar{R}(p)(x)))=_V g(\bar{R}(p)(x)) =_V f(x)$$ 
for all $x$. Now the third proof object gives similarly $h(\bar{R}(r)(x)) =_V f(x)$ for all $x$. Hence for all $x$,
$$\bar{R}(q)(\bar{R}(p)(x)) \sim_h \bar{R}(r)(x).$$
Thus $\bar{R}$ is a family of setoids over $(V,=_V)$.
\end{proof}

From the family $(V,\bar{R})$, we may construct the category ${\mathcal C}={\mathcal C}(V,\bar{R})$, as in Section  \ref{catfromfam}
 and, then compare it to the category $\mathcal V$ above. The objects of the two categories are give by the same setoid. 
 Let $F_0:{\mathcal C}_0 \to {\mathcal V}_0$ be the identity map. There is a
bijection
${\mathcal C}_1 \to {\mathcal V}_1$
given by 
$$(a,b,f) \mapsto \langle \langle a, b \rangle, \gamma(|f|,{\rm ext}_f) \rangle.$$
Further, this yields a bijection $F_2: {\mathcal C}_2 \to {\mathcal V}_2$ by letting $F_1$
act on the two component arrows. It is then straightforward to verify that $F_0$, $F_1$ and $F_2$
form a functor which is an isomorphism. We have

\begin{thm} \label{main}
The categories ${\mathcal C}(V,\bar{R})$ and ${\mathcal V}$ are isomorphic. \qed
\end{thm}

\subsection{CZFU -- constructive sets with urelements} \label{CZFU_section}

The model $V(M)$ in Section \ref{VM} suggests  an axiomatization of CZF with urelements or atoms. 
For an example of a classical set theory with atoms, see e.g.\  \cite{Mo}.
In \cite{A}, a theory called ${\rm CZF}^{\rm I}$, which is CZF extended with a class of individuals, is 
mentioned but the axioms are not detailed in that paper. It is not clear to us whether it is actually a version of the theory presented below. Nevertheless, we propose the following axiomatization of CZF with atoms, CZFU.

The language is that of set theory, with a binary predicate for membership $\in$, extended with unary
predicate $\rm S$, for being a set. Define ${\rm A}(x)= \lnot {\rm S}(x)$. Write $\forall^{\rm S} x ...$ for $\forall x. {\rm S}(x) \Rightarrow ... $ and  
$\exists^{\rm S} x ...$ for $\exists x. {\rm S}(x) \land ... $.

The axioms are the following
\begin{enumerate}[label={\cW 0}(C\arabic*)]
\item $\forall x.{\rm S}(x) \lor {\rm A}(x)$.  Each object is either a set or an atom.

\item $\forall xy. y \in x \Rightarrow {\rm S}(x)$. An object which has an element must be a set.

\item $\forall^{\rm S} x. \forall^{\rm S} y. (\forall z. z \in x \iff z \in y) \Rightarrow x=y$. Sets are determined by their elements.

\item Let $\varphi(x)$ be any formula. Then take set-induction for this formula as an axiom
$$(\forall x. (\forall y \in x. \varphi(x)) \Rightarrow \varphi(x))
\Rightarrow \forall x. \varphi(x).$$
\end{enumerate}
Since atoms have no elements this is actually equivalent to
$$(\forall x. {\rm A}(x) \Rightarrow \varphi(x)) \Rightarrow (\forall^{\rm S} x. (\forall y \in x. \varphi(x)) \Rightarrow \varphi(x))
\Rightarrow \forall x. \varphi(x).$$

\begin{enumerate}[resume,label={\cW 0}(C\arabic*)]
\item Union: $\forall^{\rm S} x. \exists^{\rm S} u. (\forall z. z \in u \iff (\exists y \in x) z \in y)$.

\item Pairing: $\forall x y.\exists^{\rm S} u.(\forall z. z \in u \iff (z =x \lor z=y))$.

\item Bounded separation: Let $\varphi(x)$ be any bounded formula. Then take as an axiom:
$$\forall^{\rm S} u. \exists^{\rm S} v. \forall x. x \in v \iff x \in u \land \varphi(x).$$ 

\item Subset collection: for any formula $\vf$
\begin{equation*}
\begin{split}
\forall a b.\exists^{\rm S} c.\forall u.& (\forall x \in a.\exists y \in b. \vf(x,y,u)) \Rightarrow  \\
&\exists d \in c.  (\forall x \in a.\exists y \in d. \vf(x,y,u)) \land (\forall y \in d.\exists x \in a. \vf(x,y,u))
\end{split}
\end{equation*}

\item Strong collection:  for any formula $\vf$
$$\forall a. (\forall x \in a.\exists y.\vf(x,y)) \Rightarrow 
\exists^S b.  (\forall x \in a.\exists y \in b. \vf(x,y)) \land (\forall y \in b.\exists x \in a. \vf(x,y))
$$

\item[(C10)] Infinity axiom: $$\exists^{\rm S} x. \emptyset \in x \land (\forall y \in x) y^+ \in x.$$
Here $y^+ = \{y, \{y\}\}$.
\end{enumerate}

\noindent If we add the purity axiom (everything is a set) we get a system, which is easily seen to be equivalent to the standard CZF.

\medskip
(Purity): $\forall x. S(x)$.

\begin{thm} \label{CZFUthm}
For any setoid $M=(|M|,=_M)$ belonging to $U$, the set-theoretic universe $V(M)$ is a model of CZFU. 
The model also verifies that there is a set containing all atoms, that is
\begin{equation} \label{atomsformset}
\exists^{\rm S} x. \forall z. z \in x \Longleftrightarrow {\rm A}(z).
\end{equation}
\end{thm}
\begin{proof} The proof is similar to the verification in Aczel's standard set-theoretic model in case of the axioms
C3 -- C6, C8 -- C10. The axioms C1 and C2 are directly verified by the meaning of ${\rm A}$ and $\in_V$. As for axiom C7, bounded
separation, we may use the standard proof once we have noticed that by Lemma \ref{smallness}, $a=_Vb$ and $a \in_V b$ are
in $U$, whenever $M$ is in $U$.

To verify (\ref{atomsformset}) first construct $a=\sup(m,f)$ where $m:U$ is such that $T(m)=M$ and $f:M \to V$ is given by
$f(t) = {\rm atom}(t)$. Then for any $z\in V$, $z \in_V a$ if, and only if, there is $t: T(m) $ such that $z=_V {\rm atom}(t)$, that is
${\rm A}(t)$ is true.

\end{proof}

\section{The implementation in Coq and  applications} \label{coq_implementation}

In our Coq implementation \cite{PW12} we understand setoids in the sense of propositions-as-types, which means that the equality relation takes its truth values in {\tt Set}  or {\tt Type}. This is
in contrast \changenote{More common word ...} to the standard setoids of Coq where the equality relation is {\tt Prop}-valued.
(Cf.\ {\em total setoids} of Barthe {\em et al.} \cite{BCP}.)
We have used the built-in type {\tt Set} to interpret the universe $U$. The setoids belonging to $U$ are therefore setoids based on  {\tt Set} and
called just {\tt setoids}. What we call setoids in this paper is called {\tt Typeoid} in the Coq code and they are based on {\tt Type}. 

The $V$-sets and $V(M)$-sets are constructed using the generalized inductive definitions available for  {\tt Type} of  Coq.  They could as well have been constructed using a general W-type.
In several places record types are used, which corresponds to $\Sigma$-type applications of MLTT.
The following theorems of the paper are formalized: Theorems \ref{cafthm}, \ref{generator}.(d), 
\ref{haspullbacks}, \ref{V_is_cat}, \ref{main},  and \ref{CZFUthm}.

We verify as well the Regular Extension Axiom (REA) \cite{A2} in our Coq implementation. This axiom is crucial for formalizing transfinite inductive definitions in CZF. 
There are important extensions of the REA  \cite{LR} that unfortunately seem difficult to model in the Coq-system, since the system currently lacks the ability to handle general inductive-recursive definition. 

Apart from using a set-theoretic universe to overcome the limitations of the built in equality of intensional type theory, there are also possible practical utilizations in proof development. An application of the kind of implementation presented here is to integrate type-theoretic and set-theoretic methods in proofs. The set-theoretic methods make it possible avoid certain coherence problems that may be difficult to solve in type theory, for instance regarding families of setoids and involved inductive definitions. The type-theoretic methods have the well known advantages with type checking that guide construction of proofs, and a direct computational interpretation. \changenote{Added further motivation}One may develop theorems in CZF (or CZFU) and then translate the first order formulas and proofs into the richer language that is modelled in the Coq implementation. This translation can easily be done automatically, and the development of the CZF theorems could be done in a theorem prover or proof assistant that can handle intuitionistic logic.


\begin{thebibliography}{99}
\bibitem{A} Peter Aczel. The type-theoretic interpretation of constructive set theory. In: A.\ Macintyre, L.\ Pacholski and J.\ Paris (eds.), {\em Logic Colloquium '77. } North-Holland, Amsterdam 1978.

\bibitem{A2} Peter Aczel. The type-theoretic interpretation of constructive set theory: inductive definitions. In: R.B. Marcus, G.J. Dorn, and G.J.W. Dorn (eds.), {\em Logic, Methodology, and Philosophy of Science VII,}  North-Holland, Amsterdam and New York, 1986,  pp.\ 17 -- 49. \changenote{Added reference}

\bibitem{BCP} Gilles Barthe, Venanzio Capretta and Olivier Pons. Setoids in type theory. {\em Journal of Functional Programming} 13(2003), pp. 261--293.

\bibitem{BvdB} Benno van den Berg and Ieke Moerdijk. A unified approach to Algebraic Set Theory. In: S.B.\ Cooper, H.\ Geuvers, A.\ Pillay and J.\ V\" a\"an\"anen (eds.) {\em Logic Colloquium 2006,} Lecture Notes in Logic, Cambridge University Press 2009,
 pp.\ 18 -- 37.

\bibitem{Ca} Aurelio Carboni. Some free constructions in realizability and proof theory. {\em Journal of Pure and Applied Algebra} 103 (1995), pp.\ 117 -- 148. \changenote{Added reference}

\bibitem{Hi} Jason J.\ Hickey. {\em The MetaPRL Logical Programming Environment.} PhD thesis, Cornell University, Ithaca, NY, January 2001.

\bibitem{H94} Martin Hofmann. On the Interpretation of Type Theory in Locally Cartesian Closed Categories
In: {\em Proceedings of Computer Science Logic,} Lecture Notes in Computer Science, Springer, 1994, pp.\ 427 -- 441.


\bibitem{HS} Martin Hofmann and Thomas Streicher. The groupoid
interpretation of type theory.  In: G.\ Sambin and J.\ Smith (eds.) {\em Twenty-five years of constructive type
theory (Venice, 1995),} pp.\ 83 -- 111, Oxford Logic Guides, 36, Oxford
Univ. Press, New York, 1998.

\bibitem{LR} Robert S.\ Lubarsky and Michael Rathjen. On the regular extension axiom and its variants. {\em Mathematical Logic Quarterly}
49 (2003), pp.\ 513 -- 518.

\bibitem{MacL} Saunders Mac Lane. {\em Categories for the Working Mathematician. 2nd ed.} Springer 1997.  \changenote{Added reference}

\bibitem{Ma} Maria E.\ Maietti. Modular correspondences between dependent type theories
and categorical universes including pretopoi and topoi. {\em Mathematical Structures in Computer Science }15 (2005), pp.\ 1089 --1149.  \changenote{Added reference}

\bibitem{Me} Nax P.\  Mendler. {\em Note: An Implementation of Constructive Set Theory in the LEGO System.} Department of Computer Science, Manchester University 1991.

\bibitem{MP} Ieke Moerdijk and Erik Palmgren.  Type Theories, Toposes and Constructive Set Theory: Predicative Aspects of AST.
{\em Annals of Pure and Applied Logic} 114(2002), pp.\ 155 -- 201.

\bibitem{Mo} Yannis N.\  Moschovakis. {\em Notes on Set Theory. Second Edition.} Springer 2006.


\bibitem{P} Erik Palmgren. Proof-relevance of families of setoids and identity in type theory. {\em Archive for Mathematical Logic} 51(2012), pp.\ 35 -- 47.

\bibitem{P12} Erik Palmgren. Constructivist and Structuralist Foundations: Bishop's and Lawvere's Theories of Sets. {\em Annals of Pure and Applied Logic } 163(2012), pp.\ 1384 -- 1399.

\bibitem{PW12} Erik Palmgren and Olov Wilander.   {\tt www.math.su.se/\~{ }palmgren/coq/czf\_and\_setoids} File repository of the implementation described in the present paper.

\bibitem{St93} Thomas Streicher. Investigations into intensional type theory. Habilitation Thesis, Ludwig-Maximilians
Universit\"at, Munich, 1993. {\tt http://www.mathematik.tu-darmstadt.de/\~{ }streicher/} 

\bibitem{T} Makoto Takeyama. Personal communication.

\bibitem{W} Olov Wilander. Constructing a small category of setoids. {\em Mathematical Structures in Computer Science} 22(2012), pp.\ 103 -- 121.

\bibitem{Y} Xin Yu. {\em Formalizing abstract algebra in constructive set theory.} Master's thesis, California Institute of Technology, 2002.

\end{thebibliography}
\end{document}